\documentclass[letterpaper,journal]{IEEEtran}
\usepackage{amsmath,amsfonts}
\usepackage{algorithm}
\usepackage{algpseudocode}
\usepackage{array}
\usepackage[caption=false,font=normalsize,labelfont=sf,textfont=sf]{subfig}
\usepackage{graphicx}
\usepackage{caption}
\usepackage{subcaption}
\usepackage{booktabs}
\usepackage{bm}
\usepackage{xcolor, colortbl, booktabs}
\usepackage[colorlinks=true, linkcolor=blue]{hyperref}

\newtheorem{problem}{Problem}

\newtheorem{theorem}{Theorem}
\newtheorem{lemma}{Lemma}

\newtheorem{assumption}{Assumption}

\newenvironment{proof}{\par\noindent\textit{Proof.}\ }{\hfill$\square$\par}

\usepackage{color}
\usepackage{textcomp}
\usepackage{stfloats}
\usepackage{url}
\usepackage{verbatim}
\usepackage{graphicx}
\usepackage{cite}
\usepackage{xcolor}

\begin{document}

\title{Distributed MPC For Coordinated Path-Following}

 \author{
 Lusine Poghosyan$^{*}$$^\ddagger$$^{\S}$, 
Anna Manucharyan$^\ddagger$, Mikayel Aramyan$^\dagger$, Naira Hovakimyan$^\dagger$ and Tigran Bakaryan$^{*}$$^\ddagger$$^{\S}$\\

$^\S$Yerevan State University, Armenia
\\
$^\ddagger$Center For Scientific Innovation and Education, Armenia\\
$^{*}$ Institute of Mathematics of National Academy of Sciences of Republic of Armenia, Armenia\\
$^\dagger$Department of Mechanical Science \& Engineering, University of Illinois Urbana-Champaign, USA

\thanks{The research is supported by the Higher Education and Science Committee of the MESCS RA under Research Project No.~24IRF-1A001 and  in part by AFOSR grant \#FA9550-21-1-0411 and NASA grant \#80NSSC22M0070.}
}

\maketitle

\begin{abstract}
In this paper, we consider a distributed model predictive control (MPC) algorithm for coordinated path-following. Relying on the time-critical cooperative path-following framework, which decouples space and time and reduces the coordination problem to a one-dimensional setting, we formulate a distributed MPC scheme for time coordination.
Leveraging properties of the normalized Laplacian, we decouple the MPC dynamics into independent modes and derive a recursive relation linking current and predicted states. Using this structure, we prove that, for prediction horizon 
$K=1$ and a fixed connected communication network, the system is exponentially stable even in the presence of path-following errors. This provides a first result on the convergence analysis of discrete-time distributed MPC schemes within this framework.

The proposed approach enables scalable and efficient real-time implementation with low communication overhead. Moreover, in contrast to the time-critical cooperative path-following framework, the optimization-based structure relaxes the reliance on preplanning by allowing the incorporation of mission-specific requirements, such as vehicle limitations, collision avoidance, and conflict resolution. Simulation results demonstrate applicability to complex scenarios, highlighting agility and exponential convergence under communication failures.
\end{abstract}


\section{Introduction}
\IEEEPARstart{C}{}oordination of multiple unmanned vehicles and robots plays a key role in modern robotics applications, including cooperative mapping, inspection, surveillance, and search-and-rescue missions \cite{olfati2007consensus, cao2013overview}. 
In these missions, each vehicle must be carefully synchronized in time to achieve overall mission success. Ensuring such temporal coordination becomes particularly challenging when the vehicles operate under communication limitations, physical constraints, and uncertain environments, without reliable and efficient access to a central controller. To address these challenges, we consider a distributed MPC-based approach.

The coordination problem investigated in this work relies on the well-known concept of time-critical cooperative path-following \cite{time-clock,5160564,7065327,kaminer2017time}, which decouples space and time in the general problem formulation, see Section \ref{sec:pre}. This separation allows the coordination challenge to be divided into three distinct subproblems: desired trajectory generation or mission planning,  path-following (effective low-level controller) and  mission execution (coordination). The advantage of this framework is that these subproblems can be addressed independently. The first, trajectory generation (see, e.g., \cite{Dubins1957OnCO, 5980409}), can be solved offline while accounting for vehicle dynamics. Path-following is an independent component implemented onboard each UAV (see, e.g., \cite{5717652, 7963104}), which is typically very efficient and possesses desirable properties such as exponential convergence. Therefore, these two components do not affect real-time implementation and do not require significant computational resources; only the real-time operation of the final subproblem matters. The final subproblem, mission execution, is a one-dimensional coordination problem that does not involve vehicle dynamics and is therefore independent of the specific vehicle and can be solved efficiently. Given these advantages, we adopt this framework for multi-vehicle coordination, with particular focus on the coordination component, which we address using a distributed MPC approach. Although the framework is vehicle-independent, in this paper we focus on multirotor UAVs.

The cooperative path-following problem has been extensively studied in the context of proportional–integral (PI) controllers, with applications to both marine vehicles \cite{VANNI200775} and UAVs \cite{kaminer2017time}. The asymptotic stability of these methods has been established under various assumptions on the connectivity of the communication network \cite{kaminer2017time, KH-23}. These approaches are efficient and applicable in real-world scenarios due to their low communication cost and computational efficiency. Nevertheless, they have a significant limitation: their strong reliance on preplanning restricts adaptability and reduces UAV autonomy, particularly when satisfying operational requirements or responding to time-varying mission objectives. To overcome these limitations, in \cite{paper}, we reformulate the coordination problem within a game-theoretic framework (see Fig. \ref{fig:block-diagram}). This approach offers greater generality and flexibility in accommodating operational constraints and dynamically evolving mission goals, while preserving the advantages of low communication and computational cost. In this paper, we investigate the numerical realization of this game-theoretic framework through a distributed MPC formulation. Note that existing stability analyses in the cooperative path-following framework (see, for example, \cite{kaminer2017time} and references therein) are conducted in continuous time. To the best of our knowledge, this is the first work that directly studies the convergence properties of a discrete-time numerical method in this setting.

As mentioned above, using the concept of virtual time (see Section~\ref{sec:pre}), we formulate the coordination problem as a one-dimensional distributed MPC (DMPC) problem (see Section~\ref{sec:mpc}) with a strongly convex quadratic cost. 
To analyze the convergence of the DMPC scheme, we consider an auxiliary unconstrained MPC problem (see Section~\ref{sec:unconstraint}), where only the system dynamics are retained and constraints arising from vehicle limitations (such as velocity and acceleration bounds) are removed. This setting allows us to derive an explicit expression for the MPC solution. 
Leveraging this representation, together with spectral properties of the normalized Laplacian, we show that, for prediction horizon $K=1$, the closed-loop dynamics decouple into independent modes. In particular, we establish an explicit relation between consecutive MPC steps via a $2\times 2$ matrix that depends only on the Laplacian eigenvalues (capturing the network connectivity) and the time step. 
Using this explicit recurrence relation and Gelfand’s formula, we prove that the associated matrix is a contraction under a fixed and connected communication network, which implies convergence of the unconstrained DMPC scheme, Theorem~\ref{theorem-unconstraint}. 
Finally, we show that for sufficiently small time step $h$, there exists a domain such that, for all initial conditions in this domain, the unconstrained solution satisfies the constraints of the original DMPC problem, see Theorem~\ref{theorem-main}. Consequently, the DMPC scheme is exponentially convergent, and the convergence rate is explicitly characterized. The analysis also accounts for disturbances arising from path-following errors. 

To the best of our knowledge, this is the first approach in the literature that proves exponential stability of a DMPC scheme through an explicit modal analysis, while also providing an explicit convergence rate. Existing convergence results have generally been established through Lyapunov- and invariance-based arguments, or through geometric properties of optimal trajectories; see, for example, Zhou and Li~\cite{zhou2015distributed} and Ferrari-Trecate et al.~\cite{ferrari2009consensus}.

The optimization-based formulation enables direct incorporation of physical and mission-level constraints while preserving convergence guarantees, meaning the control inputs provided by our method are always feasible. In contrast, PI controller–based methods (see \cite{kaminer2017time}) require problem-dependent parameter tuning, which becomes increasingly challenging as the number of vehicles and mission complexity grow. For a detailed comparison, see \cite{paper} (Version~2).

Although the analysis assumes a fixed communication network, the proposed algorithm performs well under communication failures. Environmental disturbances enter as path-following errors and do not affect convergence due to the strong convexity of the one-dimensional DMPC formulation. These properties are evident in the simulations of \cite{paper}.

Here, we further validate the proposed method through simulations in RotorPy \cite{folk2023rotorpy}, an open-source simulator for multirotor UAVs. We study the effect of the prediction horizon $K$ and observe results consistent with our theoretical findings, in particular exponential stability for $K=1$, while also demonstrating real-time feasibility and scalability. In particular, the computational time remains nearly unchanged as the number of UAVs increases.
The results show that $K=1$ achieves the best computational performance, with an average update rate of approximately 100~Hz and an order-of-magnitude improvement over $K=5$ and $K=10$ (see Table~\ref{tab:mpc_comparison}). However, shorter horizons lead to slower coordination.

In addition to this sensitivity analysis, we consider a challenging scenario that further motivates the proposed game-theoretic approach. In this scenario, two modes are considered: pure synchronization, and navigation through a narrow corridor that permits only single-UAV passage while avoiding inter-agent collisions. 
To address this problem, we consider two formulations: an offline ordering based on preplanning, and an autonomous game-based approach that determines the order during the mission (see Section~\ref{sec:seqvsgame}). 
In the presence of disturbances, preplanning-based methods may fail, as disturbances can alter the prescribed ordering. In contrast, the proposed approach yields efficient and adaptive coordination by incorporating operational objectives, ensuring collision-free passage of the group. This behavior is demonstrated in Section~\ref{sec:seqvsgame}.

The proposed distributed MPC algorithm combines the scalability and low communication requirements of consensus-based methods with the predictive and constraint-handling capabilities of MPC, making it suitable for real-time implementation in complex coordination missions. 
We establish exponential stability of the discrete-time scheme and provide an explicit convergence rate under a fixed connected communication network and prediction horizon $K=1$.  Simulation results demonstrate scalability and effectiveness in complex coordination scenarios.

\begin{figure}
    \centering
\includegraphics[width=1.0\linewidth]{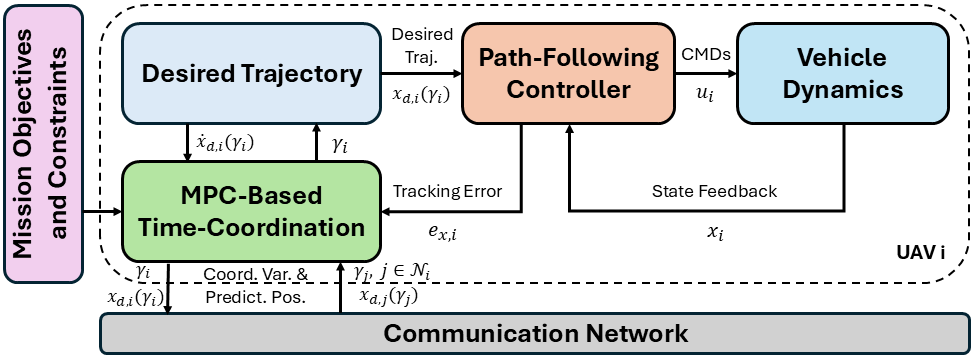}
    \caption{Cooperative path-following control framework of multi-agent UAV systems.}
    \centering
    \label{fig:block-diagram}
\end{figure}

\section{Preliminaries}\label{sec:pre}

In this section, we review the time-critical cooperative path-following framework and formulate the coordination problem within this setting. In particular, we recall the notion of virtual time and the associated variables used throughout the paper.

The cooperative path-following framework (see \cite{doi:10.1137/060678993, kaminer2017time}) comprises three main components: a trajectory-generation algorithm (see the light-blue block in Fig.~\ref{fig:block-diagram}) that produces a set of feasible and collision-free paths, a path-following controller (orange block in Fig.~\ref{fig:block-diagram}) that enables UAVs to follow assigned trajectories, and a time-coordination control algorithm (green block in Fig.~\ref{fig:block-diagram}) that allows UAVs to adjust their pace to maintain coordination with neighboring agents. 

Although the framework can be applied to heterogeneous multi-vehicle systems, in this paper we focus on multirotor UAVs. Section~\ref{path-planning-and-assumptions} introduces the notion of virtual time and discusses the constraints associated with trajectory generation. Section~\ref{time-coordination-objective} defines the coordination objective in the virtual-time domain.

Before introducing virtual time, we briefly review two key subproblems in cooperative path-following: mission planning and path-following.
Mission planning is achieved through trajectory generation. For each UAV, trajectories can be designed using optimal-control formulations, waypoint-based schemes, or minimum-snap techniques commonly used in quadrotor motion planning (see \cite{Beard2012, 4310229, Dubins1957OnCO, 5980409}).
Each UAV is equipped with a path-following controller (e.g., PID controllers \cite{salih2010flight}, geometric controllers \cite{5717652}, and adaptive path-followers \cite{7963104}) to ensure accurate tracking of its assigned trajectory.

\subsection{Virtual Time}\label{path-planning-and-assumptions}
 
We introduce the notation and definitions required to formulate the coordination problem in the virtual-time domain.
Consider a mission involving a group of $N$ UAVs, where $N \in \mathbb{N}$. Each UAV $i$ is assigned a desired trajectory generated during the mission planning stage. The desired trajectory is defined as
$
x_{d,i} : [0, t^*_{d,i}] \rightarrow \mathbb{R}^3,
$
where $t^*_{d,i}$ denotes the desired completion time for UAV $i$. These trajectories satisfy the linear velocity and acceleration constraints given by:
\begin{equation} \label{eq:dyn_constraints}
\begin{split}
0 \leq v^i_{\min} < v^i_{d,\min} \leq \left\| \dot{x}_{d,i}(t) \right\| \leq v^i_{d,\max} < v^i_{\max}, \\
\left\| \ddot{x}_{d,i}(t) \right\| \leq a^i_{d,\max} < a^i_{\max},
\end{split}
\end{equation}
where $v^i_{\min}$ and $v^i_{\max}$ are the UAV’s physical speed limits, and $a^i_{\max}$ is the maximum acceleration. The mission-specific bounds $v^i_{d,\min}, v^i_{d,\max}, a^i_{d,\max}$ are chosen to ensure feasible and safe motion. Moreover, throughout the mission, a minimum separation distance $E$ is enforced for safety: \begin{equation}\label{def-sp-sep}
\min _{\substack{i, j=1, \ldots, N \\ i \neq j}}\left\|x_{d, i}(t)-x_{d, j}(t)\right\|^2 \geq E^2>0, \,\text{for all } t\geq 0. 
\end{equation} 

In this work, we focus on the simultaneous-arrival mission, as other time-critical mission specifications can be analyzed using similar arguments. Accordingly, a common nominal final time $t_d^* \in \mathbb{R}^+$ is assigned to all UAVs; that is,
\begin{equation}
t_{d,i}^* = t_d^*, \quad i = 1, \ldots, N.
\end{equation}
This ensures synchronized arrival under ideal conditions.

Virtual time is defined as a mapping $\gamma_i: \mathbb{R}^{+} \rightarrow [0, t_d^*]$ (see \cite{7065327,kaminer2017time}), which encodes the progress of the $i$th UAV along its nominal mission timeline. By mapping real time to this shared timeline, $\gamma_i$ acts as a coordination variable that regulates deviations in mission progress across the UAVs. The resulting time-parameterized trajectory $x_{\gamma,i}(t)=x_{d,i}(\gamma_i(t))$ is then used as the reference input for the path-following controller, defining the overall structure of the method (see Fig.~\ref{fig:block-diagram}).

We denote by $\gamma_i^0 := \gamma_i(0)$ the initial virtual time of UAV $i$, which represents its initial progress along the nominal mission timeline. For convenience, we assume $\gamma_i^0 \geq 0$ for all $i = 1, \dots, N$, corresponding to defining the mission start time as the instant when the last UAV begins execution. Other initialization choices can be handled similarly without affecting the formulation. 

The mission of UAV $i$ is completed when its virtual time reaches the terminal value $t_d^*$ at some time $t_f^i$; that is, $\gamma_i(t_f^i)=t_d^*$.

Note that $\dot{\gamma}_i(t) > 1$ corresponds to an accelerated progression along the path, $\dot{\gamma}_i(t) < 1$ indicates a decelerated progression, and $\dot{\gamma}_i(t) = 1$ implies that the UAV follows the desired trajectory at the nominal pace. The  physical constraints in \eqref{eq:dyn_constraints}  impose limits on the  virtual time function:
\begin{equation}\label{gammadotconst}
\tfrac{v_{\min }^i}{v_{d, \min }^i} \leq \dot{\gamma}_{i,\min} \leq \dot{\gamma}_i(t) \leq \dot{\gamma}_{i,\max} \leq \tfrac{v_{\max }^i}{v_{d, \max }^i},
\end{equation}
\begin{equation}\label{gammaddotconst}
\ddot{\gamma}_{\max }^i v_{d, \max }^i+\left(\dot{\gamma}_{\max }^i\right)^2 a_{d, \max }^i \leq a_{\max }^i, \, |\ddot{\gamma}_i(t)| \leq \ddot{\gamma}_{i,\max}
\end{equation}
For a detailed derivation, see \cite{kaminer2017time} or \cite{cichella20133d}.

\subsection{Time-Coordination}\label{time-coordination-objective}

Depending on the mission specifications, the coordination problem in terms of virtual time can take different forms. In this work, we focus on the synchronization (simultaneous-arrival) case, for which we provide a theoretical analysis. Accordingly, we present only the formulation of the synchronization problem here. Other mission formulations, such as sequential arrival and competitive scenarios, are illustrated in the simulation section, see Section~\ref{sec:sim}.

The goal is to achieve simultaneous arrival of all UAVs while preserving the desired speed profiles, without strictly enforcing the nominal arrival time $t_d^*$.
This requirement is expressed as a coordination condition on the virtual time variables:
\begin{equation}\label{def-syn} \gamma_i(t) - \gamma_j(t) = 0, \quad \forall \, i, j \in \{1, 2, \dots, N\}. 
\end{equation}
Additionally, each UAV should maintain its nominal progression rate along the trajectory:
\begin{equation}\label{def-pace-keep} \dot{\gamma}_i(t) = 1, \quad \forall \,i \in \{1, 2, \dots, N\}. \end{equation}

\section{Main Approach}\label{sec:mpc}

This section describes the MPC-based algorithm for distributed time coordination of UAVs, originally introduced in \cite{paper}.
The algorithm computes \emph{virtual-time} trajectories that enforce the synchronization and pace-keeping objectives defined in \eqref{def-syn} and \eqref{def-pace-keep}.

Recall that pace keeping requires
\[
\dot\gamma_i(t)\to 1,
\]
while synchronization requires
\[
\gamma_i(t)-\gamma_j(t)\to 0.
\]
These objectives motivate describing the coordination dynamics in terms of the deviation variables 
\[
\delta_{i}(t)=\gamma_i(t)-t,\qquad
\dot\delta_{i}(t)=\dot\gamma_i(t)-1,\qquad
\ddot\delta_i(t)=\ddot\gamma_i(t),
\]
which represent, respectively, the deviation of the virtual time from its nominal evolution and the deviation of its rate from the desired unit value.  In terms of these new variables, the time-coordination objectives are
\[
\dot\delta_i(t)\to 0, \qquad \delta_i(t)-\delta_j(t)\to 0.
\]

As in standard DMPC formulations,  at each discrete time instant \(t_k=kh\), UAV \(i\) computes a predicted deviation trajectory by solving a local finite-horizon optimal control problem. The optimization is carried out over a prediction horizon of length \(K\) and uses the decision variables
\[
\bm{\delta}^k_i = (\delta^k_{i,0}, \dots, \delta^k_{i,K})^T,
\qquad
\dot{\bm{\delta}}^k_i = (\dot\delta^k_{i,0}, \dots, \dot\delta^k_{i,K})^T,
\]
\[
\bm{u}^k_i = (u^k_{i,0}, \dots, u^k_{i,K-1})^T.
\]

The index \(\tau=0\) corresponds to the current sampling instant \(t_k\), whereas the indices \(\tau=1,\dots,K\) correspond to predicted future values within the finite-horizon optimization problem. 
The coordination state is described by the deviation variables $\delta_{i,0}^k$, $\dot\delta_{i,0}^k$, while $u_{i,0}^k$
denotes the control input computed at time \(t_k\). For $\tau>0$, \(\delta_{i,\tau}^k\) and \(\dot\delta_{i,\tau}^k\) denote the corresponding predicted deviation variables. Only the first input \(u^k_{i,0}\) is applied to the system, while the remaining components are auxiliary prediction variables used to compute the optimal control action.

The algorithm is initialized at time \(t_0=0\), so the initial variables are
\[
\delta^0_{i,0}=\gamma_i^0-t_{0}=\gamma_i^0,
\qquad
\dot{\delta}^0_{i,0}=\dot{\gamma}_i^0-1.
\]
At $k=0$, before any optimization is performed, the predicted deviation trajectory is initialized by propagating the initial state with zero input:
\[
\delta^0_{i,\tau+1}=\delta^0_{i,\tau}+h\dot{\delta}^0_{i,\tau},
\quad
\dot{\delta}^0_{i,\tau+1}=\dot{\delta}^0_{i,\tau},\quad \tau=0,\dots,K-1.
\]
For \(k\ge 1\), the initial values \(\delta^k_{i,0}\) and \(\dot{\delta}^k_{i,0}\) are determined from the optimal prediction computed at the previous time step according to the MPC update rule.

Using the bounds \eqref{gammadotconst} and \eqref{gammaddotconst} on the virtual-time rate and acceleration, we define the admissible set associated with
\[
\mathbf{z}^{k}_{i} = (\bm{\delta}^{k}_{i}, \dot{\bm{\delta}}^{k}_{i}, \bm{u}^{k}_{i})
\]
as
\begin{equation}
\begin{split}
A_i^k = \Big\{ \mathbf{z}^k_i:\;
&\ \dot{\delta}_{i,\min}
\le
\dot{\delta}^k_{i,\tau}
\le
\dot{\delta}_{i,\max},
\quad \tau = 0,\dots,K,
\\
&
|u^k_{i,\tau}|
\le
\ddot{\gamma}_{i,\max},
\quad \tau = 0,\dots,K-1
\Big\},
\end{split}
\end{equation}
where 
\[
\dot{\delta}_{i,\min}=\dot{\gamma}_{i,\min} - 1,
\qquad
\dot{\delta}_{i,\max}=\dot{\gamma}_{i,\max} - 1.
\]

The local finite-horizon optimization problem of UAV $i$ at time $t_k$ is given by
\begin{equation}\label{ProblemK}
\begin{aligned}
\min_{\mathbf{z}^{k}_{i}\in A_{i}^k} \quad 
& J_i\!\left(\mathbf{z}^{k}_{i}, \bm{\delta}^{k-1}_{-i}\right)
\\
\text{subject to} \quad
& \delta_{i,\tau+1}^k
=
\delta_{i,\tau}^k
+
h\dot\delta_{i,\tau}^k
+
\dfrac{h^2}{2}u_{i,\tau}^k,
\\[3pt]
& \dot\delta_{i,\tau+1}^k
=
\dot\delta_{i,\tau}^k
+
hu_{i,\tau}^k,
\\[3pt]
& \delta^{k}_{i,0}
=
\delta^{k-1}_{i,1}
-
\alpha_{i}^{k},
\\[3pt]
& \dot{\delta}^{k}_{i,0}
=
\dot{\delta}^{k-1}_{i,1},
\end{aligned}
\end{equation}
for $k=1,2,\dots$ and $\tau = 0,\dots,K-1$, where $\bm{\delta}_{-i}^{k-1}:= \left[\bm{\delta}^{k-1}_{1},\dots,\bm{\delta}^{k-1}_{i-1},\bm{\delta}^{k-1}_{i+1},\dots,\bm{\delta}^{k-1}_{N}\right]$
denotes the predicted virtual-time deviation sequences received from the other UAVs.

The  constraints in \eqref{ProblemK} relate the prediction variables across the horizon, enforce the discrete-time double-integrator dynamics in deviation coordinates, and implement the MPC shift together with the path-following correction.

The cost function $J_i$ is defined as follows:
\begin{equation}\label{cost1}
	J_i\left(\mathbf{z}^{k}_{i},\bm{\delta}^{k-1}_{-i}\right)=((F_i^{cm}\left(\bm{\delta}^{k}_{i},\bm{\delta}^{k-1}_{-i}\right),\|\dot{\bm{\delta}}^k_i\|_2^2,
	\|\bm{u}_i^k\|_2^2),\bm{w})
\end{equation}
with $\bm{w}=(w_1,w_2,w_3)$ and 
\begin{equation}\label{synchron}
	\begin{split}
	F_i^{cm}\left(\bm{\delta}^{k}_{i},\bm{\delta}^{k-1}_{-i}\right)=\frac{1}{\left|\mathcal{N}_{i,k}\right|}\sum_{j\in \mathcal{N}_{i,k}} \|\bm{\delta}_i^k-\bm{\delta}_j^{k-1}\|_2^2.
	\end{split}
\end{equation}
Here, $\mathcal{N}_{i,k}$ denotes the set of UAVs that communicate with UAV $i$ at time $t_{k}$.
The first term of the cost function penalizes deviations between the predicted  trajectory (solution) of UAV $i$ and the most recently received deviation trajectories of its neighbors, thereby promoting synchronization. The second term penalizes deviations from the desired unit pace, while the third term penalizes excessive control effort. The positive weights $w_1, w_2, w_3$  determine the relative emphasis of each component.

Path-following errors that arise in uncertain and dynamic environments
are incorporated into the initial condition of the optimization problem enabling the UAV to decelerate to reduce forward overshoot or to accelerate in order to catch up with the reference trajectory.
In particular, we introduce the correction term $\alpha_i^k$ defined as
\begin{equation}\label{def-a}
	\alpha_{i}^{k} =\alpha_{i}^{k}(x_{i}(t_{k}))= \beta\dfrac{(x_{\gamma,i}(t_{k})-x_{i}(t_{k}))^T\dot{x}_{\gamma,i}(t_{k})}{\left|\left|\dot{x}_{\gamma,i}(t_{k})\right|\right|+\delta},
\end{equation}
where $\beta$, $\delta$ are positive parameters and $x_{i}(t_{k})$ represents $i$-th UAV  position at time $t_{k}$.  $\alpha_{i}^{k}$ is negative when the projection of the UAV's actual position onto the desired trajectory lies ahead of the reference position $x_{\gamma_i}(t_{k})$ and positive otherwise.
The parameter $\delta$ in the correction term $\alpha_{i}^{k}$ is introduced to prevent division by zero; it can be set as a small positive constant. The parameter $\beta$ regulates the influence of the path-following error on the coordination process. Increasing $\beta$ accelerates time coordination but may cause constraint violations if chosen excessively large.

The algorithm accounts for both time-varying communication topology and path-following errors. Although the network topology may change from one time step to the next, it is assumed to remain fixed throughout each optimization interval.

At each time step $t_k$, each UAV solves the local problem in \eqref{ProblemK}
using the most recent information $\bm{\delta}_{-i}^{k-1}$ received from its neighbors,
and transmits the resulting virtual-time deviation sequence $\bm{\delta}_i^{k}$
to the UAVs in $\mathcal N_{i,k+1}$ at the next sampling instant $t_{k+1}$.
This distributed MPC procedure, summarized in Algorithm~\ref{alg:algorithm} that
together with the initialization at $k=0$, generates a sequence
$\{\mathbf z_i^k\}_{k=0}^{\infty}$.

\begin{algorithm}[t]
\caption{Distributed MPC for multi-agent time coordination}
\label{alg:algorithm}
\begin{algorithmic}[1]
\State \textbf{Initialize:} number of UAVs $N$, desired trajectories $x_{d,i}$, bounds $\dot{\gamma}_{i,\min}$, $\dot{\gamma}_{i,\max}$, $\ddot{\gamma}_{i,\max}$, initial conditions $\gamma_i^0$, $\dot{\gamma}_i^0$, prediction horizon $K$, final time $T$, sampling period $h$
\State Set $\delta_{i,0}^0 \gets \gamma_i^0,\,
\dot{\delta}_{i,0}^0 \gets \dot{\gamma}_i^0-1,\, i=1,\dots,N$
\State $k \gets 1$
\While{$t_{k}= kh \le T$}
    \For{$i = 1$ to $N$}
        \State Measure the actual position $x_i(t_k)$
        \State Compute the path-following correction $\alpha_i^{k}=\alpha_i^{k}(x_i(t_{k}))$
        \State Set the initial conditions for the local MPC problem:
        \[
        \delta^{k}_{i,0} \gets \delta^{k-1}_{i,1} - \alpha_{i}^{k},
        \qquad
        \dot{\delta}^{k}_{i,0} \gets \dot{\delta}^{k-1}_{i,1}
        \]
        \State Solve the local optimization problem \eqref{ProblemK} using
        $\delta^{k}_{i,0}$, $\dot{\delta}^{k}_{i,0}$, and $\bm{\delta}_{-i}^{k-1}$,
        obtaining $\bm{\delta}_i^k,\,
        \dot{\bm{\delta}}_i^k,\,
        \bm{u}_i^k$
        \State Transmit $\bm{\delta}_i^k$ to the UAVs in $\mathcal N_{i,k+1}$
    \EndFor
    \State $k \gets k+1$
\EndWhile
\end{algorithmic}
\end{algorithm}
Under suitable assumptions on the communication network and the path-following error, the virtual-time deviations synchronize and the rate deviations converge to zero; that is,
\[
(\delta_{i,0}^k - \delta_{j,0}^k) \to 0,
\qquad
\dot{\delta}_{i,0}^k \to 0,
\quad \text{as } k \to \infty.
\]
In the next section, for the prediction horizon $K=1$, we prove the convergence.

\section{Convergence Proof}

In this section, we establish convergence of the proposed DMPC scheme under a fixed communication topology with an undirected and connected communication graph, and for prediction horizon \(K=1\).

We first analyze an unconstrained DMPC problem, obtained by removing from the admissible set \(A_i^k\) the bounds on \(\dot{\delta}\) and the control input \(u\). For this simplified problem, we prove convergence of the resulting scheme.

We then show that, for sufficiently small time step \(h\) and for a suitable class of initial conditions, the solution of the unconstrained problem satisfies the constraints of the original DMPC formulation at every step. Consequently, for such initial conditions, the convergence result extends to the constrained problem.

Now, we state the assumptions used in the convergence analysis.The assumptions on the communication network are summarized as follows.
\begin{assumption}\label{ass:graph}
The communication topology is fixed, i.e.,
\[
\mathcal N_{i,k}=\mathcal N_i,\qquad i=1,\dots,N,\quad k\ge 0,
\]
and the associated communication graph is undirected and connected.
\end{assumption}

In addition, we assume that each UAV is equipped with an asymptotically stable path-following controller.
\begin{assumption}\label{ass:tracking}
The path-following correction satisfies
\[
|\alpha_i^k|\le d e^{-\nu t_k}=d e^{-\nu kh},
\qquad i=1,\dots,N,\quad k\ge 0,
\]
for some constants \(d>0\) and \(\nu>0\).
\end{assumption}

\subsection{Unconstrained Case}\label{sec:unconstraint}
In this subsection, we study the corresponding unconstrained DMPC problem. This setting admits an explicit expression for the solution. Using this representation, we derive a recurrence relation between consecutive MPC solutions and analyze the resulting closed-loop dynamics to establish convergence.

Before stating the main result, we introduce auxiliary variables based on the explicit solution of the unconstrained problem and on the spectral properties of the Laplacian and degree matrices. This transformation decouples the system into independent modes and leads to a recurrence relation for each mode governed by a \(2\times 2\) matrix.

Since the prediction horizon is \(K=1\), the index \(\tau\) in \eqref{ProblemK} takes only the value \(0\). Therefore, at each sampling instant \(t_k\), UAV \(i\) solves the following minimization problem:
\begin{problem}\label{DiscreteUn} 
\begin{equation*}
\begin{aligned}
\min_{\mathbf{z}^{k}_{i}} \quad 
& J_i\!\left(\mathbf{z}^{k}_{i}, \bm{\delta}^{k-1}_{-i}\right)
\\
\text{subject to} \quad
\delta_{i,1}^k
&=
\delta_{i,0}^k
+
h\dot\delta_{i,0}^k
+
\dfrac{h^2}{2}u_{i,0}^k,
\\[3pt]
\dot\delta_{i,1}^k
&=
\dot\delta_{i,0}^k
+
hu_{i,0}^k,
\\[3pt]
\delta^{k}_{i,0}
&=
\delta^{k-1}_{i,1}
-
\alpha_{i}^{k},
\\[3pt]
\dot{\delta}^{k}_{i,0}
&=
\dot{\delta}^{k-1}_{i,1}.
\end{aligned}
\end{equation*}
\end{problem}

By substituting the one-step prediction dynamics into the cost function, Problem~\ref{DiscreteUn} reduces, up to additive constants independent of \(u_{i,0}^k\), to the minimization of the quadratic function
\begin{equation}\label{ProblemK1}
\begin{split}
&\bar J_i(u^k_{i,0})
=
w_2(\dot\delta^{k-1}_{i,1}+hu^k_{i,0})^2
+
w_3(u^k_{i,0})^2
\\
&+
\tfrac{w_1}{|\mathcal N_i|}
\sum_{j\in \mathcal N_i}
(
\delta^{k-1}_{i,1}
+h\dot\delta^{k-1}_{i,1}
-\alpha_i^{k}
+\tfrac{h^2}{2}u^k_{i,0}
-\delta^{k-1}_{j,1}
)^2 .
\end{split}
\end{equation}
Since \(w_1,w_2,w_3>0\), the coefficient of \((u_{i,0}^k)^2\) is strictly positive, and therefore \(\bar J_i\) is strictly convex.
The unique minimizer of the quadratic function \(\bar J_i\) on \(\mathbb R\) is given by
\begin{equation}\label{control}
u^k_{i,0}
=
-a^h\,\tfrac{1}{|\mathcal N_i|}\sum_{j\in\mathcal N_i}
(\delta^{k-1}_{i,1}-\delta^{k-1}_{j,1})
-b^h\,\dot\delta^{k-1}_{i,1}
+a^h\,\alpha_i^{k},
\end{equation}
where
\[
a^h
=
\tfrac{w_1h^2}{\,2(w_3+w_2h^2+w_1\frac{h^4}{4})\,},
\qquad
b^h
=
\tfrac{w_2h+w_1\frac{h^3}{2}}{\,w_3+w_2h^2+w_1\frac{h^4}{4}\,}.
\]

Let \(L\) be the Laplacian matrix of the fixed communication graph and let
\[
D=\operatorname{diag}(|\mathcal N_1|,\dots,|\mathcal N_N|)
\]
be the corresponding degree matrix. Since the graph is connected, \(D\) is invertible. Then the control law in \eqref{control} can be written compactly as
\begin{equation}\label{control2}
\bm{U}^k = -a^h D^{-1}L\,\bm{\Delta}^{k-1}
- b^h \dot{\bm{\Delta}}^{k-1}
+ a^h \bm{\alpha}^{k},
\end{equation}
where
\begin{equation}
\begin{aligned}
&\bm{U}^k := [u_{1,0}^k,\dots,u_{N,0}^k]^T,\quad
\bm{\Delta}^k := [\delta_{1,1}^k,\dots,\delta_{N,1}^k]^T,\\
&\dot{\bm{\Delta}}^k := [\dot\delta_{1,1}^k,\dots,\dot\delta_{N,1}^k]^T,\quad
\bm{\alpha}^k := [\alpha_1^k,\dots,\alpha_N^k]^T.
\end{aligned}
\end{equation}

The discrete-time state update for the network is
\begin{equation}\label{state}
\begin{cases}
\bm{\Delta}^{k} = \bm{\Delta}^{k-1}- \bm{\alpha}^{k}
+ h \dot{\bm{\Delta}}^{k-1}
+ \dfrac{h^2}{2} \bm{U}^k,\\[6pt]
\dot{\bm{\Delta}}^{k}  = \dot{\bm{\Delta}}^{k-1}  + h \bm{U}^k.
\end{cases}
\end{equation}

Substituting \eqref{control2} into \eqref{state} yields the closed-loop affine system
\begin{equation}\label{state2}
\begin{bmatrix}
\bm{\Delta}^{k} \\[5pt]
\dot{\bm{\Delta}}^{k}
\end{bmatrix}
=
A
\begin{bmatrix}
\bm{\Delta}^{k-1} \\[5pt]
\dot{\bm{\Delta}}^{k-1}
\end{bmatrix}
+
\begin{bmatrix}
(\tfrac{h^2}{2}a^h-1)\bm{\alpha}^{k} \\[5pt]
a^h h\,\bm{\alpha}^{k}
\end{bmatrix},
\end{equation}
where \(A\in\mathbb{R}^{2N\times 2N}\) is given by
\begin{equation}
A=
\begin{bmatrix}
I - \dfrac{a^{h}h^{2}}{2} D^{-1}L
& \left(h - \dfrac{b^{h}h^{2}}{2}\right) I \\[5pt]
-a^{h}h D^{-1}L
& \left(1- b^{h} h \right) I
\end{bmatrix}.
\end{equation}

Since the communication graph is undirected, the Laplacian matrix \(L\) is symmetric. Therefore, the normalized Laplacian
\[
\mathcal L = D^{-1/2} L D^{-1/2}
\]
is also symmetric. Although the random-walk normalized Laplacian \(D^{-1}L\) is generally not symmetric, it is similar to \(\mathcal L\), since
\[
D^{-1}L = D^{-1/2}\,\mathcal L\,D^{1/2}.
\]
Hence \(D^{-1}L\) is diagonalizable and has the same eigenvalues as \(\mathcal L\). Since \(\mathcal L\) is symmetric and positive semidefinite, all eigenvalues of \(D^{-1}L\) are real and nonnegative. Moreover, for an undirected, connected graph, these eigenvalues satisfy
\[
0=\lambda_1 < \lambda_2\le \dots \le \lambda_N\le 2
\]
(see, e.g., \cite{Biggs}).

Therefore,
\[
D^{-1}L = V\Lambda V^{-1},
\]
where \(\Lambda=\mathrm{diag}(\lambda_1,\dots,\lambda_N)\), and the columns of \(V\) are eigenvectors of \(D^{-1}L\). The eigenvector corresponding to \(\lambda_1=0\) is the constant vector
\[
\mathbf 1=[1,\dots,1]^T.
\]

Since $D^{-1}L = V\Lambda V^{-1}$, we introduce the transformed variables
\begin{equation}\label{varsub}
\bar{\bm{\Delta}}^k = V^{-1}\bm{\Delta}^k,
\quad
\dot{\bar{\bm{\Delta}}}^k = V^{-1}\dot{\bm{\Delta}}^k,
\quad
\bar{\bm{\alpha}}^k = V^{-1}\bm{\alpha}^k .
\end{equation}
Here, $\bm{\Delta}^k,\dot{\bm{\Delta}}^k,\bm{\alpha}^k \in \mathbb{R}^N$,
and $V \in \mathbb{R}^{N\times N}$ is the matrix of eigenvectors of $D^{-1}L$.

Under this change of coordinates, the Laplacian term becomes diagonal. More precisely, applying the block-diagonal transformation
\[
\operatorname{diag}(V^{-1},V^{-1})
\]
to \eqref{state2} and using
\[
\bm{\Delta}^k = V\bar{\bm{\Delta}}^k,
\qquad
\dot{\bm{\Delta}}^k = V\dot{\bar{\bm{\Delta}}}^k,
\qquad
D^{-1}L=V\Lambda V^{-1},
\]
the closed-loop system \eqref{state2} becomes
\begin{equation}\label{state3}
\begin{bmatrix}
\bar{\bm{\Delta}}^k  \\[5pt]
\dot{\bar{\bm{\Delta}}}^k
\end{bmatrix}
=
\mathcal Q^{h}
\begin{bmatrix}
\bar{\bm{\Delta}}^{k-1}  \\[5pt]
\dot{\bar{\bm{\Delta}}}^{k-1}
\end{bmatrix}
+
\begin{bmatrix}
(\tfrac{h^2}{2}a^h-1)\,\bar{\bm{\alpha}}^{k} \\[5pt]
a^h h\,\bar{\bm{\alpha}}^{k}
\end{bmatrix},
\quad k\ge1,
\end{equation}
where
\[
\mathcal Q^{h}:=
\begin{bmatrix}
I - \dfrac{a^{h}h^{2}}{2}\Lambda 
& \left(h - \dfrac{b^{h} h^{2}}{2}\right) I \\[6pt]
- a^{h} h \Lambda  
& (1 - b^{h} h) I
\end{bmatrix}
\in \mathbb{R}^{2N\times 2N}.
\]

Iterating \eqref{state3}, we obtain
\begin{equation}\label{state4}
\begin{bmatrix}
\bar{\bm{\Delta}}^k \\[5pt]
\dot{\bar{\bm{\Delta}}}^k
\end{bmatrix}
=
\left(\mathcal Q^{h}\right)^k
\begin{bmatrix}
\bar{\bm{\Delta}}^{0} \\[5pt]
\dot{\bar{\bm{\Delta}}}^{0}
\end{bmatrix}
+
\begin{bmatrix}
\bar{\bm{v}}^{k}\\[5pt]
\bar{\bm{w}}^{k}
\end{bmatrix},
\end{equation}
where
\begin{equation}\label{barvbarw}
\begin{bmatrix}
\bar{\bm{v}}^k \\[5pt]
\bar{\bm{w}}^k
\end{bmatrix}
=
\sum_{s=1}^{k}
\left(\mathcal Q^{h}\right)^{k-s}
\begin{bmatrix}
(\tfrac{h^2}{2}a^h-1)\,\bar{\bm{\alpha}}^{s} \\[5pt]
a^h h\,\bar{\bm{\alpha}}^{s}
\end{bmatrix}.
\end{equation}

In the eigenvector coordinates of \(D^{-1}L\), the transformed system decouples into \(N\) independent \(2\times2\) subsystems. Hence, for each \(i=1,\dots,N\), the pair
\[
(\bar{\delta}_{i,1}^k,\dot{\bar{\delta}}_{i,1}^k)
\]
evolves according to a decoupled \(2\times 2\) subsystem associated with the eigenvalue \(\lambda_i\). More precisely, \eqref{state4} yields
\begin{equation}\label{state5}
\begin{bmatrix}
\bar{\delta}_{i,1}^k \\[5pt]
\dot{\bar{\delta}}_{i,1}^k
\end{bmatrix}
=
(Q^{h}_{i})^k
\begin{bmatrix}
\bar{\delta}_{i,1}^{0} \\[5pt]
\dot{\bar{\delta}}_{i,1}^{0}
\end{bmatrix}
+
\begin{bmatrix}
\bar{v}_i^k \\[5pt]
\bar{w}_i^k
\end{bmatrix},
\end{equation}
where
\[
Q^{h}_{i}=
\begin{bmatrix}
1-\dfrac{h^2}{2}a^h\lambda_{i} & h-\dfrac{h^2}{2} b^h \\[6pt]
- h a^h\lambda_{i} & 1- h b^h
\end{bmatrix},
\quad
B^h :=
\begin{bmatrix}
\dfrac{a^h h^2}{2}-1 \\[6pt]
a^h h
\end{bmatrix},
\]
and \(\bar{\bm{v}}^{k}=[\bar{v}_{1}^{k},\dots,\bar{v}_{N}^{k}]^T\),
\(\bar{\bm{w}}^{k}=[\bar{w}_{1}^{k},\dots,\bar{w}_{N}^{k}]^T\) are defined componentwise by
\begin{equation}\label{vk_def}
\begin{bmatrix}
\bar{v}_i^k \\[5pt]
\bar{w}_i^k
\end{bmatrix}
=
\sum_{s=1}^{k}
(Q_i^h)^{k-s}\, B^h\, \bar{\alpha}_i^{s},
\qquad k=1,2,\dots .
\end{equation}

Having obtained the recurrence relation in \eqref{state5}, we analyze the matrices defining these relations. The following result shows that, for sufficiently small time step $h$, the spectral radius of the matrices associated with the positive eigenvalues of \(D^{-1}L\) (i.e., for $i = 2, \dots, N$) is strictly less than $1$.

\begin{lemma}\label{lemma:rho}
For every eigenvalue \(\lambda_i>0\) of \(D^{-1}L\), there exists \(h_i>0\) such that, for all
\(h\in(0,h_i)\),
\[
\rho(Q_i^h)<1,
\]
where \(\rho(Q_i^h)\) denotes the spectral radius of \(Q_i^h\).
\end{lemma}
\begin{proof} 
Since \(Q_i^h \in \mathbb{R}^{2\times 2}\), its eigenvalues can be computed explicitly as the roots of the characteristic polynomial
\[
\mu^2 - \operatorname{tr}(Q_i^h)\mu + \det(Q_i^h) = 0.
\]
By use the exact dependence of eigenvalues on the time step \(h\), we obtain the  result.
So, the eigenvalues are 
\[
\mu_{\pm}(h)
=
\frac12\left(
\operatorname{tr}(Q_i^h)\pm
\sqrt{\operatorname{tr}(Q_i^h)^2-4\det(Q_i^h)}
\right).
\]
Moreover,
\[
\operatorname{tr}(Q_i^h)=2-hb^h-\frac{h^2}{2}a^h\lambda_i,
\]
and
\begin{equation}\label{det_formula_lemma}
\det(Q_i^h)
=
1-hb^h+\frac{h^2}{2}a^h\lambda_i
=
\frac{w_3+\frac{w_1h^4}{4}(\lambda_i-1)}
{w_3+w_2h^2+\frac{w_1h^4}{4}}.
\end{equation}
Also,
\begin{equation}\label{discriminant_lemma}
\operatorname{tr}(Q_i^h)^2-4\det(Q_i^h)
=
\left(hb^h+\frac{h^2}{2}a^h\lambda_i\right)^2-4h^2a^h\lambda_i.
\end{equation}
There are two possible cases for the eigenvalues of  \(Q_i^h \in \mathbb{R}^{2\times 2}\). If
\[
\operatorname{tr}(Q^{h}_{i})^2 - 4 \det(Q^{h}_{i}) <0,
\]
then the eigenvalues of $Q^{h}_{i}$ form a complex-conjugate pair
$\mu,\bar\mu$. Hence, by Vieta's theorem
\[
\rho(Q^{h}_{i})^2 = |\mu|^2 = \mu\bar\mu = \det(Q^{h}_{i}).
\]
Since \(\lambda_i\in(0,2]\), we have \(\lambda_i-1\le1\), and thus by \eqref{det_formula_lemma},
\[
\det(Q_i^h)
\le
\frac{w_3+\frac{w_1h^4}{4}}
{w_3+w_2h^2+\frac{w_1h^4}{4}}
<1.
\]
Since the eigenvalues form a complex-conjugate pair, one has
\[
\det(Q_i^h)=|\mu|^2>0.
\]
Therefore, from \(\det(Q_i^h)<1\), it follows that
\[
\rho(Q_i^h)=|\mu|=\sqrt{\det(Q_i^h)}<1.
\]
In the second case,
\begin{equation}
    \label{eq:ei2}
    \operatorname{tr}(Q_i^h)^2-4\det(Q_i^h)\ge 0,
\end{equation}
then the eigenvalues of $Q_i^h$ are real.  Since \(Q_i^h\to I\) as \(h\to0\), both real eigenvalues converge to \(1\). Hence, for \(h>0\) sufficiently small, both eigenvalues are positive, and therefore the spectral radius coincides with the larger eigenvalue:
\[
\rho(Q_i^h)
=
\frac12\left(
\operatorname{tr}(Q_i^h)+
\sqrt{\operatorname{tr}(Q_i^h)^2-4\det(Q_i^h)}
\right).
\]
Using \eqref{discriminant_lemma}, we obtain
\[
\rho(Q_i^h)
=
1-\frac12\eta_h
\left(
1-\sqrt{1-\frac{4h^2a^h\lambda_i}{\eta_h^2}}
\right),
\]
where
\[
\eta_h:=hb^h+\frac{h^2}{2}a^h\lambda_i.
\]
Since \(a^h>0\), \(b^h>0\), and \(\lambda_i>0\), we have \(\eta_h>0\). Moreover,  by \eqref{eq:ei2}, we get
\[
0<
\frac{4h^2a^h\lambda_i}{\eta_h^2}
\le1,
\]
which implies
\[
1-\sqrt{1-\frac{4h^2a^h\lambda_i}{\eta_h^2}}>0.
\]
Therefore, \(\rho(Q_i^h)<1\).
\end{proof}


\begin{theorem}\label{theorem-unconstraint}
Suppose that Assumptions~\ref{ass:graph}--\ref{ass:tracking} hold. Then, for sufficiently small \(h>0\), there exist
\[
r_h\in(0,1),\qquad M_0^h>0,\qquad \dot M_0^h>0,
\]
such that, for all \(k\ge 1\),
\begin{equation}\label{eq:teo1}
    |\delta_{i,0}^k-\delta_{j,0}^k|
\le M_0^h r_h^k,
\,\ |\dot{\delta}_{i,0}^{k}|
\le \dot M_0^h r_h^k,\,\ i,j=1,\dots,N.
\end{equation}
In particular,
\[
\delta_{i,0}^k-\delta_{j,0}^k\to 0,
\qquad
\dot{\delta}_{i,0}^k\to 0,
\quad \text{as } k\to\infty.
\]
\end{theorem}
\begin{proof}
We first establish convergence results for the auxiliary variables defined in \eqref{varsub}.
In particular, we distinguish two cases based on the spectrum of \(D^{-1}L\): strictly positive eigenvalues and the zero eigenvalue. For the first case, Lemma~\ref{lemma:rho} implies exponential convergence of the corresponding variables to zero. For the second case, we show convergence to a constant.
Combining these results and mapping them back to the original variables concludes the proof.

Since the communication graph is undirected and connected, the eigenvalues of \(D^{-1}L\) satisfy
\[
0=\lambda_1<\lambda_2\le \dots \le \lambda_N\le 2.
\]
Hence, for \(i=2,\dots,N\), one has \(\lambda_i>0\), and Lemma~\ref{lemma:rho} implies that, for sufficiently small \(h>0\),
\[
\rho(Q_i^h)<1,\qquad i=2,\dots,N.
\]
Fix such an \(h>0\), and choose \(r_h\in(0,1)\) such that
\[
\rho(Q_i^h)<r_h<1,\,\, i=2,\dots,N,
\,\,
r_h>e^{-\nu h},
\,\,
r_h>1-hb^h.
\]
By Gelfand's formula,
\[
\lim_{k\to\infty}\|(Q_i^h)^k\|_\infty^{1/k}
=
\rho(Q_i^h)
<r_h,
\qquad i=2,\dots,N.
\]
Therefore, for each \(i=2,\dots,N\), there exists a constant \(C_{i,\infty}^h>0\) such that
\[
\|(Q_i^h)^k\|_\infty \le C_{i,\infty}^h\, r_h^k,
\qquad \forall k\in\mathbb N.
\]
Since the set \(\{2,\dots,N\}\) is finite, defining
\[
C_\infty^h:=\max_{i=2,\dots,N} C_{i,\infty}^h,
\]
we obtain
\[
\|(Q_i^h)^k\|_\infty \le C_\infty^h\, r_h^k,
\qquad \forall k\in\mathbb N,\quad i=2,\dots,N.
\]

For \(i=2,\dots,N\), relation \eqref{state5} gives
\[
\begin{bmatrix}
\bar{\delta}_{i,1}^k \\[5pt]
\dot{\bar{\delta}}_{i,1}^k
\end{bmatrix}
=
(Q_i^h)^k
\begin{bmatrix}
\bar{\delta}_{i,1}^{0} \\[5pt]
\dot{\bar{\delta}}_{i,1}^{0}
\end{bmatrix}
+
\begin{bmatrix}
\bar{v}_i^k \\[5pt]
\bar{w}_i^k
\end{bmatrix}.
\]
Using the bound above, \eqref{vk_def}, submultiplicativity of induced norms, we obtain
\[
\left\|
\begin{bmatrix}
\bar{\delta}_{i,1}^k \\[5pt]
\dot{\bar{\delta}}_{i,1}^k
\end{bmatrix}
\right\|_\infty
\le
C_\infty^h r_h^k
\left\|
\begin{bmatrix}
\bar{\delta}_{i,1}^{0} \\[5pt]
\dot{\bar{\delta}}_{i,1}^{0}
\end{bmatrix}
\right\|_\infty
+
C_\infty^h\|B^h\|_\infty
\sum_{s=1}^{k}
r_h^{k-s}|\bar{\alpha}_i^s|.
\]
Moreover, by  Assumption~\ref{ass:tracking} for \(i=1,\dots,N\),
\[
|\bar{\alpha}_i^s|
\le \|\bar{\bm{\alpha}}^s\|_\infty
\le \|V^{-1}\|_\infty \|\bm{\alpha}^s\|_\infty
\le \|V^{-1}\|_\infty d\,e^{-\nu sh}.
\]
Therefore,
\[
\begin{split}
\left\|
\begin{bmatrix}
\bar{\delta}_{i,1}^k \\[5pt]
\dot{\bar{\delta}}_{i,1}^k
\end{bmatrix}
\right\|_\infty
&\le
C_\infty^h r_h^k \|V^{-1}\|_\infty
\max\{\|\bm{\Delta}^0\|_\infty,\|\dot{\bm{\Delta}}^0\|_\infty\}
\\
&\quad+
C_\infty^h\|B^h\|_\infty \|V^{-1}\|_\infty d
\sum_{s=1}^{k}
r_h^{k-s}e^{-\nu sh}.
\end{split}
\]
Since \(r_h>e^{-\nu h}\),
\[
\sum_{s=1}^{k}
r_h^{k-s}e^{-\nu sh}
=
r_h^k\sum_{s=1}^{k}\left(\frac{e^{-\nu h}}{r_h}\right)^s
\le
r_h^k\frac{e^{-\nu h}}{r_h-e^{-\nu h}}.
\]
Hence, for all \(i=2,\dots,N\) and all \(k\ge 0\),
\begin{equation}\label{eq:disagreement_mode_bound}
\left\|
\begin{bmatrix}
\bar{\delta}_{i,1}^k \\[5pt]
\dot{\bar{\delta}}_{i,1}^k
\end{bmatrix}
\right\|_\infty
\le
A_1^h
\max\{\|\bm{\Delta}^0\|_\infty,\|\dot{\bm{\Delta}}^0\|_\infty\}\,r_h^k
+
A_2^h\, d\, r_h^k,
\end{equation}
where
\[
A_1^h:=C_\infty^h\|V^{-1}\|_\infty,
\,\,
A_2^h:=C_\infty^h\|B^h\|_\infty \|V^{-1}\|_\infty
\frac{e^{-\nu h}}{r_h-e^{-\nu h}}.
\]

We now consider the case \(i=1\), corresponding to \(\lambda_1=0\). In this case,
\[
Q_1^h=
\begin{bmatrix}
1 & h-\dfrac{h^2}{2}b^h\\[6pt]
0 & 1-hb^h
\end{bmatrix},
\]
so \(Q_1^h\in\mathbb R^{2\times2}\) has two distinct eigenvalues
\[
\mu_1=1,
\qquad
\mu_2=q:=1-hb^h,
\qquad 0<q<1.
\]
Its characteristic polynomial is
\[
\chi(x)=(x-1)(x-q).
\]
By the Cayley--Hamilton theorem,
\[
(Q_1^h)^2-(1+q)Q_1^h+qI=0.
\]
Hence, for every \(k\in\mathbb N\), \((Q_1^h)^k\) can be written in the form
\[
(Q_1^h)^k=c_k Q_1^h+d_k I.
\]
Imposing this identity at the eigenvalues \(1\) and \(q\) gives
\[
c_k+d_k=1,
\qquad
c_k q+d_k=q^k,
\]
and therefore
\[
(Q_1^h)^k
=
\frac{1}{1-q}(Q_1^h-qI)+\frac{q^k}{1-q}(I-Q_1^h).
\]
Setting
\[
P:=\frac{1}{1-q}(Q_1^h-qI),
\qquad
F:=\frac{1}{1-q}(I-Q_1^h),
\]
we have
\begin{equation}\label{eq:Q_power_decomp_thm_refined}
(Q_1^h)^k=P+q^kF.
\end{equation}
Because $0<q<1$, we deduce that
\[
\lim_{k\to\infty}(Q_1^h)^k=P.
\]
Since the second row of \(P\) is zero, \eqref{eq:Q_power_decomp_thm_refined} implies
\[
\begin{split}
&\left|
\left[
(Q_1^h)^k
\begin{bmatrix}
\bar{\delta}_{1,1}^{0}\\[5pt]
\dot{\bar{\delta}}_{1,1}^{0}
\end{bmatrix}
\right]_2
\right|
\le
q^k\|F\|_\infty
\left\|
\begin{bmatrix}
\bar{\delta}_{1,1}^{0}\\[5pt]
\dot{\bar{\delta}}_{1,1}^{0}
\end{bmatrix}
\right\|_\infty
\\&\quad\le
r_h^k\|F\|_\infty\|V^{-1}\|_\infty
\max\{\|\bm{\Delta}^0\|_\infty,\|\dot{\bm{\Delta}}^0\|_\infty\}.
\end{split}
\]
Also, substituting \eqref{eq:Q_power_decomp_thm_refined} into \eqref{vk_def}, we obtain
\[
\begin{bmatrix}
\bar{v}_1^k \\[5pt]
\bar{w}_1^k
\end{bmatrix}
=
PB^h\sum_{s=1}^{k}\bar{\alpha}_1^s
+
FB^h\sum_{s=1}^{k}q^{k-s}\bar{\alpha}_1^s.
\]
The first term converges as \(k\to\infty\), since \(\sum_{s=1}^\infty \bar{\alpha}_1^s\) converges absolutely, while the second component of the first term is identically zero. Moreover,
\[
\sum_{s=1}^{k} q^{k-s}|\bar{\alpha}_1^s|
\le
\|V^{-1}\|_\infty d
\sum_{s=1}^{k} q^{k-s}e^{-\nu sh}
\le C_q^h r_h^k
\]
for some constant \(C_q^h>0\), because \(q<r_h\) and \(e^{-\nu h}<r_h\). Therefore, there exists \(\widetilde M^h>0\) such that
\[
|\dot{\bar{\delta}}_{1,1}^{k}|
\le
\widetilde M^h r_h^k.
\]

Returning to the original coordinates, since
\[
\dot{\bm{\Delta}}^k = V\dot{\bar{\bm{\Delta}}}^k,
\]
we obtain
\[
|\dot{\delta}_{i,1}^{k}|
\le
\|V\|_\infty \max_{s=1,\dots,N}|\dot{\bar{\delta}}_{s,1}^{k}|,
\qquad i=1,\dots,N.
\]
Using \eqref{eq:disagreement_mode_bound} for \(s=2,\dots,N\) and the previous estimate for \(s=1\), we conclude that there exists \(\dot M^h>0\) such that
\begin{equation}\label{eq:rate_bound_index1}
|\dot{\delta}_{i,1}^{k}|
\le
\dot M^h r_h^k,
\qquad i=1,\dots,N.
\end{equation}

Next, since \(\bar{\delta}_{1,1}^k\) is common to all agents in the reconstruction of \(\bm{\Delta}^k\), and the first column of \(V\) is \(\mathbf 1=[1,\dots,1]^T\), we have
\[
\delta_{i,1}^k-\delta_{j,1}^k
=
\sum_{s=2}^{N}(v_{i,s}-v_{j,s})\,\bar{\delta}_{s,1}^{k},
\qquad i,j=1,\dots,N.
\]
Hence, by \eqref{eq:disagreement_mode_bound},
\[
\begin{split}
|\delta_{i,1}^k-\delta_{j,1}^k|
&\le
\sum_{s=2}^{N}|v_{i,s}-v_{j,s}|\,|\bar{\delta}_{s,1}^{k}|
\\
&\le
2\|V\|_\infty \max_{s=2,\dots,N} |\bar{\delta}_{s,1}^{k}|
\\
&\le
M_1^h r_h^k,
\qquad i,j=1,\dots,N,
\end{split}
\]
for some constant \(M_1^h>0\).

Finally, using the shift conditions
\[
\delta_{i,0}^k=\delta_{i,1}^{k-1}-\alpha_i^{k},
\qquad
\dot\delta_{i,0}^k=\dot\delta_{i,1}^{k-1},
\qquad k\ge 1,
\]
we obtain
\[
|\dot\delta_{i,0}^k|
=
|\dot\delta_{i,1}^{k-1}|
\le
\dot M^h r_h^{k-1}
=
\frac{\dot M^h}{r_h}\,r_h^k,
\qquad i=1,\dots,N,
\]
and
\[
\begin{split}
|\delta_{i,0}^k-\delta_{j,0}^k|
&\le
|\delta_{i,1}^{k-1}-\delta_{j,1}^{k-1}|+|\alpha_i^k|+|\alpha_j^k|
\\
&\le
M_1^h r_h^{k-1}+2d\,e^{-\nu hk}
\\
&\le
\left(\frac{M_1^h}{r_h}+2d\right) r_h^k,
\qquad i,j=1,\dots,N,
\end{split}
\]
since \(e^{-\nu h}<r_h\). Therefore, the claim follows with
\[
\dot M_0^h:=\frac{\dot M^h}{r_h},
\qquad
M_0^h:=\frac{M_1^h}{r_h}+2d.
\]
This completes the proof.
\end{proof}

\begin{lemma}\label{lem:uniform_delta_difference_bound}
Suppose that Assumptions~\ref{ass:graph}--\ref{ass:tracking} hold. Then,  for sufficiently small \(h>0\), there exist constants
\[
S_1^h>0,\qquad S_2^h>0,
\]
such that, for every \(i,j=1,\dots,N\) and every \(\ell\ge 1\),
\[
|\delta_{i,1}^{\ell}-\delta_{j,1}^{\ell}|
\le
S_1^h \max\{\|\bm{\Delta}^0\|_\infty,\|\dot{\bm{\Delta}}^0\|_\infty\}
+
S_2^h d.
\]
\end{lemma}

\begin{proof}
Under Assumptions~\ref{ass:graph}--\ref{ass:tracking}, for sufficiently small \(h>0\), the proof of Theorem~\ref{theorem-unconstraint} shows that there exist \(r_h\in(0,1)\), \(A_1^h>0\), and \(A_2^h>0\) such that for $s=2,\dots,N$,
\[
\left\|
\begin{bmatrix}
\bar{\delta}_{s,1}^\ell \\[5pt]
\dot{\bar{\delta}}_{s,1}^\ell
\end{bmatrix}
\right\|_\infty
\le
A_1^h \max\{\|\bm{\Delta}^0\|_\infty,\|\dot{\bm{\Delta}}^0\|_\infty\} r_h^\ell
+
A_2^h d\, r_h^\ell,
\]
for every \(\ell\ge 1\); see \eqref{eq:disagreement_mode_bound}. In particular,
\[
|\bar{\delta}_{s,1}^\ell|
\le
A_1^h \max\{\|\bm{\Delta}^0\|_\infty,\|\dot{\bm{\Delta}}^0\|_\infty\}
+
A_2^h d,
\quad s=2,\dots,N.
\]
Using the reconstruction formula
\[
\delta_{i,1}^\ell-\delta_{j,1}^\ell
=
\sum_{s=2}^{N}(v_{i,s}-v_{j,s})\,\bar{\delta}_{s,1}^{\ell},
\]
and arguing exactly as in the proof of Theorem~\ref{theorem-unconstraint}, we obtain
\[
\begin{split}
&|\delta_{i,1}^{\ell}-\delta_{j,1}^{\ell}|
\le
2\|V\|_\infty \max_{s=2,\dots,N} |\bar{\delta}_{s,1}^{\ell}|
\\
&\qquad\le
2\|V\|_\infty A_1^h
\max\{\|\bm{\Delta}^0\|_\infty,\|\dot{\bm{\Delta}}^0\|_\infty\}
+
2\|V\|_\infty A_2^h\, d.
\end{split}
\]
Therefore the claim holds with
\[
S_1^h:=2\|V\|_\infty A_1^h,
\qquad
S_2^h:=2\|V\|_\infty A_2^h.
\]
\end{proof}


\subsection{Original Problem}\label{sec:original}

e now return to the original constrained problem.

\begin{problem}\label{Discrete}
At each sampling instant \(t_k\), each UAV \(i\) solves the following minimization problem:
\begin{equation*}
\begin{aligned}
\min_{\mathbf{z}^{k}_{i}\in A_{i}^k} \quad 
& J_i\!\left(\mathbf{z}^{k}_{i}, \bm{\delta}^{k-1}_{-i}\right)
\\
\text{subject to} \quad
& \delta_{i,1}^k
=
\delta_{i,0}^k
+
h\dot\delta_{i,0}^k
+
\dfrac{h^2}{2}u_{i,0}^k,
\\[3pt]
& \dot\delta_{i,1}^k
=
\dot\delta_{i,0}^k
+
hu_{i,0}^k,
\\[3pt]
& \delta^{k}_{i,0}
=
\delta^{k-1}_{i,1}
-
\alpha_{i}^{k},
\\[3pt]
& \dot{\delta}^{k}_{i,0}
=
\dot{\delta}^{k-1}_{i,1},
\end{aligned}
\end{equation*}
for \(k=1,2,\dots\).
\end{problem}
We next show that, under suitable smallness assumptions, the minimizer of the unconstrained problem satisfies the constraints of Problem~\ref{Discrete} at every step. Consequently, along the resulting trajectory, the constrained and unconstrained problems coincide.
\begin{theorem}\label{theorem-main}
Suppose that Assumptions~\ref{ass:graph}--\ref{ass:tracking} hold, and assume that the initial predicted rates satisfy
\[
\dot{\delta}_{i,\min}\le \dot{\delta}^{0}_{i,1}\le \dot{\delta}_{i,\max},
\qquad i=1,\dots,N.
\]
Then, for sufficiently small \(h>0\), there exist constants \(\nu_h>0\) and \(d_h>0\) such that, if
\[
\max\{\|\bm{\Delta}^0\|_\infty,\|\dot{\bm{\Delta}}^0\|_\infty\}<\nu_h,
\qquad
0<d<d_h,
\]
then, for all \(i=1,\dots,N\) and all \(k\ge 1\),
\[
\dot{\delta}_{i,\min}\le \dot{\delta}^{k}_{i,1}\le \dot{\delta}_{i,\max},
\qquad
|u^k_{i,0}|\le \ddot{\gamma}_{i,\max},
\]
and the estimates in \eqref{eq:teo1} hold.
\end{theorem}

\begin{proof}
For each \(i=1,\dots,N\), let
\[
\dot{\bar{\delta}}_i:=\max\{|\dot{\delta}_{i,\min}|,\dot{\delta}_{i,\max}\}.
\]
Also define
\[
\eta_i^{k-1}
:=
\frac{a^h}{|\mathcal N_i|}\sum_{j\in\mathcal N_i}
\big|\delta^{k-1}_{i,1}-\delta^{k-1}_{j,1}\big|.
\]

We first derive sufficient conditions for the control and rate constraints.

From \eqref{control} and the triangle inequality,
\[
|u_{i,0}^k|
\le
\eta_i^{k-1}
+
b^h|\dot\delta^{k-1}_{i,1}|
+
a^h|\alpha_i^{k}|.
\]
Hence, if
\[
\dot{\delta}_{i,\min}\le \dot{\delta}^{k-1}_{i,1}\le \dot{\delta}_{i,\max},
\qquad
|\alpha_i^{k}|\le d,
\]
then
\[
|u_{i,0}^k|
\le
\eta_i^{k-1}
+
b^h\bar{\dot\delta}_i
+
a^h d.
\]
Therefore, a sufficient condition for
\[
|u_{i,0}^k|\le \ddot{\gamma}_{i,\max}
\]
is
\begin{equation}\label{u_constraint_suff_thm_refined}
\eta_i^{k-1}
\le
\ddot{\gamma}_{i,\max}-b^h\bar{\dot\delta}_i-a^h d.
\end{equation}

Next, using \eqref{control}, we write
\[
\dot{\delta}^{k}_{i,1}
=
\dot{\delta}^{k-1}_{i,1}
+
h\,u^k_{i,0}
=
(1-hb^h)\dot{\delta}^{k-1}_{i,1}
+
h\,\xi_i^{k-1},
\]
where
\[
\xi_i^{k-1}
:=
-\frac{a^h}{|\mathcal N_i|}\sum_{j\in\mathcal N_i}
\big(\delta^{k-1}_{i,1}-\delta^{k-1}_{j,1}\big)
+
a^h\alpha_i^{k}.
\]
Assume that \(0<hb^h<1\) and
\[
\dot{\delta}_{i,\min}\le \dot{\delta}^{k-1}_{i,1}\le \dot{\delta}_{i,\max}.
\]
Then a sufficient condition for
\[
\dot{\delta}_{i,\min}\le \dot{\delta}^{k}_{i,1}\le \dot{\delta}_{i,\max}
\]
is
\[
b^h\dot{\delta}_{i,\min}
\le
\xi_i^{k-1}
\le
b^h\dot{\delta}_{i,\max}.
\]
Moreover, if \(|\alpha_i^{k}|\le d\), then
\[
|\xi_i^{k-1}|
\le
\eta_i^{k-1}+a^h d.
\]
Therefore, a sufficient condition for the previous double inequality is
\begin{equation}\label{rate_constraint_suff_thm_refined}
\eta_i^{k-1}
\le
b^h\min\{\dot{\delta}_{i,\max},-\dot{\delta}_{i,\min}\}-a^h d.
\end{equation}

We now use Lemma~\ref{lem:uniform_delta_difference_bound}. For every \(i=1,\dots,N\) and every \(k\ge 1\),
\[
\eta_i^{k-1}
\le
a^h
\max_{j\in\mathcal N_i}
\big|\delta^{k-1}_{i,1}-\delta^{k-1}_{j,1}\big|
\]
and therefore
\[
\eta_i^{k-1}
\le
a^h\Bigl(
S_1^h\max\{\|\bm{\Delta}^0\|_\infty,\|\dot{\bm{\Delta}}^0\|_\infty\}
+
S_2^h d
\Bigr).
\]
Since \(a^h=O(h^2)\) and \(b^h=O(h)\), for sufficiently small \(h>0\) one has \(0<hb^h<1\), and the quantities
\[
\ddot{\gamma}_{i,\max}-b^h\bar{\dot\delta}_i-a^h d,
\qquad
b^h\min\{\dot{\delta}_{i,\max},-\dot{\delta}_{i,\min}\}-a^h d,
\]
are positive. Hence there exist \(\nu_h>0\) and \(d_h>0\) such that, whenever
\[
\max\{\|\bm{\Delta}^0\|_\infty,\|\dot{\bm{\Delta}}^0\|_\infty\}<\nu_h,
\qquad
0<d<d_h,
\]
both sufficient conditions \eqref{u_constraint_suff_thm_refined} and \eqref{rate_constraint_suff_thm_refined} hold for all \(i=1,\dots,N\) and all \(k\ge 1\).

Since
\[
\dot{\delta}_{i,\min}\le \dot{\delta}^{0}_{i,1}\le \dot{\delta}_{i,\max},
\qquad i=1,\dots,N,
\]
it follows by induction that
\[
\dot{\delta}_{i,\min}\le \dot{\delta}^{k}_{i,1}\le \dot{\delta}_{i,\max},
\qquad i=1,\dots,N,\quad k\ge 1,
\]
and simultaneously
\[
|u_{i,0}^k|\le \ddot{\gamma}_{i,\max},
\qquad i=1,\dots,N,\quad k\ge 1.
\]
This completes the proof.
    \end{proof}

\section{Simulations}\label{sec:sim}

We evaluate the proposed framework through a series of high-fidelity simulations, focusing on scalability, computational efficiency, practical implementation of the method, and validation of the theoretical results.
We consider multiple simulation scenarios. First, we perform a sensitivity  (see Section~\ref{sec:MPC-horizon}) analysis with respect to the MPC prediction horizon $K$ for varying numbers of UAVs. This study highlights the scalability of the proposed method and provides empirical validation of the theoretical results, in particular the exponential stability for $K = 1$.
As a second scenario, we consider a challenging coordination setting that highlights the advantages of the proposed game/optimization-based formulation (see Section~\ref{sec:seqvsgame}). 

The simulations were implemented using RotorPy\cite{folk2023rotorpy}, an open-source simulator optimized for multi-rotor research. For all test cases, the Bitcraze Crazyflie was selected as the primary aerial platform. The RotorPy framework allows for precise integration of the platform's equations of motion while accounting for critical aerodynamic effects.

All simulations were conducted on a GIGABYTE G5 MF5 laptop equipped 
with an Intel Core i7-13620H processor (16 cores) and 32 GiB of RAM, 
running Ubuntu 22.04.5 LTS.

\subsection{MPC Horizon}\label{sec:MPC-horizon}

Initially, we investigate a baseline scenario where the MPC horizon is set to one. We evaluate the coordination performance across a varying number of agents $N$, under identical environmental settings. To ensure a rigorous comparative analysis of the consensus-reaching time and the maximum MPC step calculation time, it is important to maintain consistent initial conditions. Thus, we use a fixed set of initial values $\bm\gamma_0 = [4.5, 6, 0, 3.5, 2.5, 5, 3.5, 6, 2, 4, 5.5, 1, 6, 3.5, 2]$. For a given number of agents \(N\), the first \(N\) components of this vector are assigned to the corresponding UAVs. In this way, the initial conditions of the original agents remain unchanged as the swarm size increases, while additional agents are introduced with comparable initial values in order to reduce variability in the baseline comparison.
 The reference trajectories used in this scenario are circles of varying radii, sharing a common center, with a total mission duration of 70 seconds. All agents communicate with one another throughout the mission, and no external disturbances are considered.

The simulation results presented in Table \ref{tab:mpc_comparison} imply that, while shorter horizons lead to a delay in reaching consensus, the maximum MPC calculation time observed throughout the mission decreases. Given that the MPC step is calculated at each time step $h$ for each of the $N$ UAVs throughout the mission, these incremental delays might accumulate in case of longer horizons, accounting for the increased computation time, which may be heavy for small and low-resource UAVs.

Furthermore, the results show the scalability of the approach; as the number of agents increases, consensus is reached faster, while the associated increase in computation time remains practically unchanged. 

The simulation results provide empirical support for the exponential stability of the DMPC scheme established in Theorem~\ref{theorem-main}.

\begin{figure}[ht]
    \centering
    \subfloat[]{\includegraphics[width=\columnwidth]{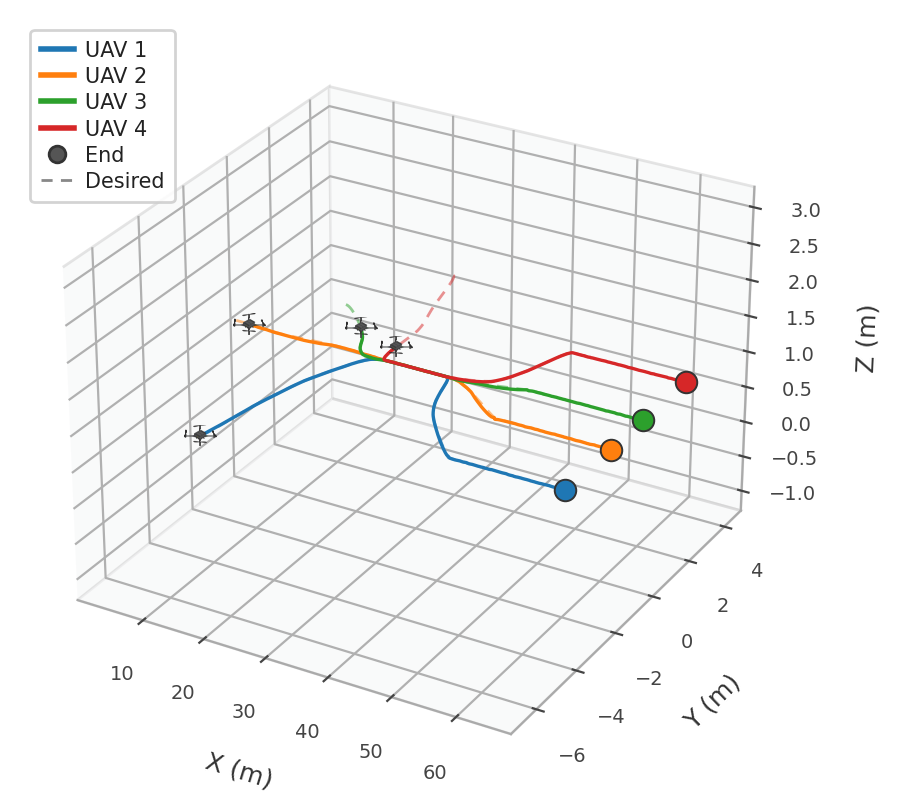}\label{fig:top_figure}}

    \caption{Agent trajectories. UAV icons represent the initial positions of the agents.}
    \label{fig:trajectories}
\end{figure}

\begin{table*}[!t]
\centering
\caption{Consensus time and maximum MPC step computation time for varying prediction horizons and number of agents.}
\label{tab:mpc_comparison}
\setlength{\aboverulesep}{0pt}
\setlength{\belowrulesep}{0pt} 
\begin{tabular}{lcccccccccccc}\toprule
\rowcolor{gray!20}
& \multicolumn{4}{c}{\textbf{$N=5$ Agents}} 
& \multicolumn{4}{c}{\textbf{$N=10$ Agents}} 
& \multicolumn{4}{c}{\textbf{$N=15$ Agents}} \\
\cmidrule(lr){2-5} \cmidrule(lr){6-9} \cmidrule(lr){10-13}
\rowcolor{gray!10}
\textbf{} 
  & \textbf{$K=1$} & \textbf{$K=5$} & \textbf{$K=10$} & \textbf{$K=25$} 
  & \textbf{$K=1$} & \textbf{$K=5$} & \textbf{$K=10$} & \textbf{$K=25$} 
  & \textbf{$K=1$} & \textbf{$K=5$} & \textbf{$K=10$} & \textbf{$K=25$} \\
\midrule
\noalign{\vspace{3pt}}
Consensus Time (s)      & 26.95 & 6.55 & 5.45 & 4.85 & 11.35 & 5.7 & 4.8 & 3.75 & 6.35 & 5.4 & 4.45 & 3.65 \\
\rowcolor{gray!5}
MPC Solve Time (s) & 0.0083 & 0.0423 & 0.0101 & 0.0215 & 0.0055 & 0.0102 & 0.0100 & 0.0177 & 0.0060 & 0.0110 & 0.0106 & 0.0186 \\
\bottomrule
\end{tabular}
\end{table*}

\subsection{Corridor Navigation: Predefined and Autonomous Passage Ordering}\label{sec:seqvsgame}
In this scenario, we highlight the flexibility and adaptability of the proposed approach by incorporating operational objectives directly into the cost function, thereby enabling the UAVs to autonomously navigate and pass through a narrow corridor that permits only one UAV at a time.
We consider two cost functions to address this problem.
The first relies on offline predefined ordering of the UAVs, while the second enables the UAVs 
to determine the passage order autonomously during the mission through a game-theoretic competition, where each UAV competes to pass first.

The agent trajectories in 3D space are illustrated in Fig.~\ref{fig:trajectories}, where UAV icons 
denote the initial position of each agent and the dashed line represents the desired 
reference trajectory over the 100-second mission. The mission start times are set to $\bm\gamma_0 = [0, 3, 7, 31]$ for UAV~1 through UAV~4, respectively. The agents have $z$-axis 
oscillations throughout the flight, except while passing the corridor. Between the 35th and 55th seconds of the mission, the agents must navigate through the corridor, requiring sequential entry to avoid inter-agent collisions,
and temporary incoordination. Upon exiting the tunnel, the agents must establish coordination.

\subsubsection{Offline Ordering via Predefined Priority}
The first modified objective function penalizes each drone $i$ for deviating from a desired separation $\Delta$ relative to its neighbors during a pre-specified phase of the mission. An initial ordering is used to order the drones offline. Drones with a lower index ($j < i$) should maintain a relative lead, while those with a higher index ($j > i$) should remain behind. This ordering logic is embedded into the cost function \eqref{cost1} by redefining  $F_i^{cm}\left(\bm{\delta}^{k}_{i},\bm{\delta}^{k-1}_{-i}\right)$ as follows:
\vspace{-8pt}
\begin{equation}
\begin{split}
    F_i^{cm}\left(\bm{\delta}^{k}_{i},\bm{\delta}^{k-1}_{-i}\right)&=\sum_{\tau=1}^{K} \Biggl[ 
    \left(1 - \mathcal{I}_{i\tau}^k\right)
     \sum_{j=1,\, j\neq i}^{N} 
    \left(\delta_{i\tau}^k - \delta_{j\tau}^{k-1}\right)^2 \\
    &+\mathcal{I}_{i\tau}^k
     \sum_{j=1,\, j\neq i}^{N} \left(\delta_{i\tau}^k - \delta_{j\tau}^{k-1} + (i-j)\Delta\right)^2 
    \Biggr],
\end{split}
\label{eq:Fi1}
\end{equation}
where $\mathcal{I}_{i\tau}^k$ is defined as
\begin{equation}\label{eq:indicator}
    \mathcal{I}_{i\tau}^k :=\mathbf{1}\bigl[\gamma_1^* \leq \delta_{i\tau}^k + t_k\leq \gamma_2^*\bigr]_,
\end{equation}
with $\mathbf{1}[\cdot]$ denoting the indicator function. $\gamma_1^*$ and $\gamma_2^*$ are predetermined values that define the interval during which the separation should be maintained. The parameter $\Delta$ prescribes the ordering of the UAVs and the magnitude of their separation. 

\begin{figure}[ht]
\centering
\begin{tabular}{cl}
    \includegraphics[width=0.35\textwidth]{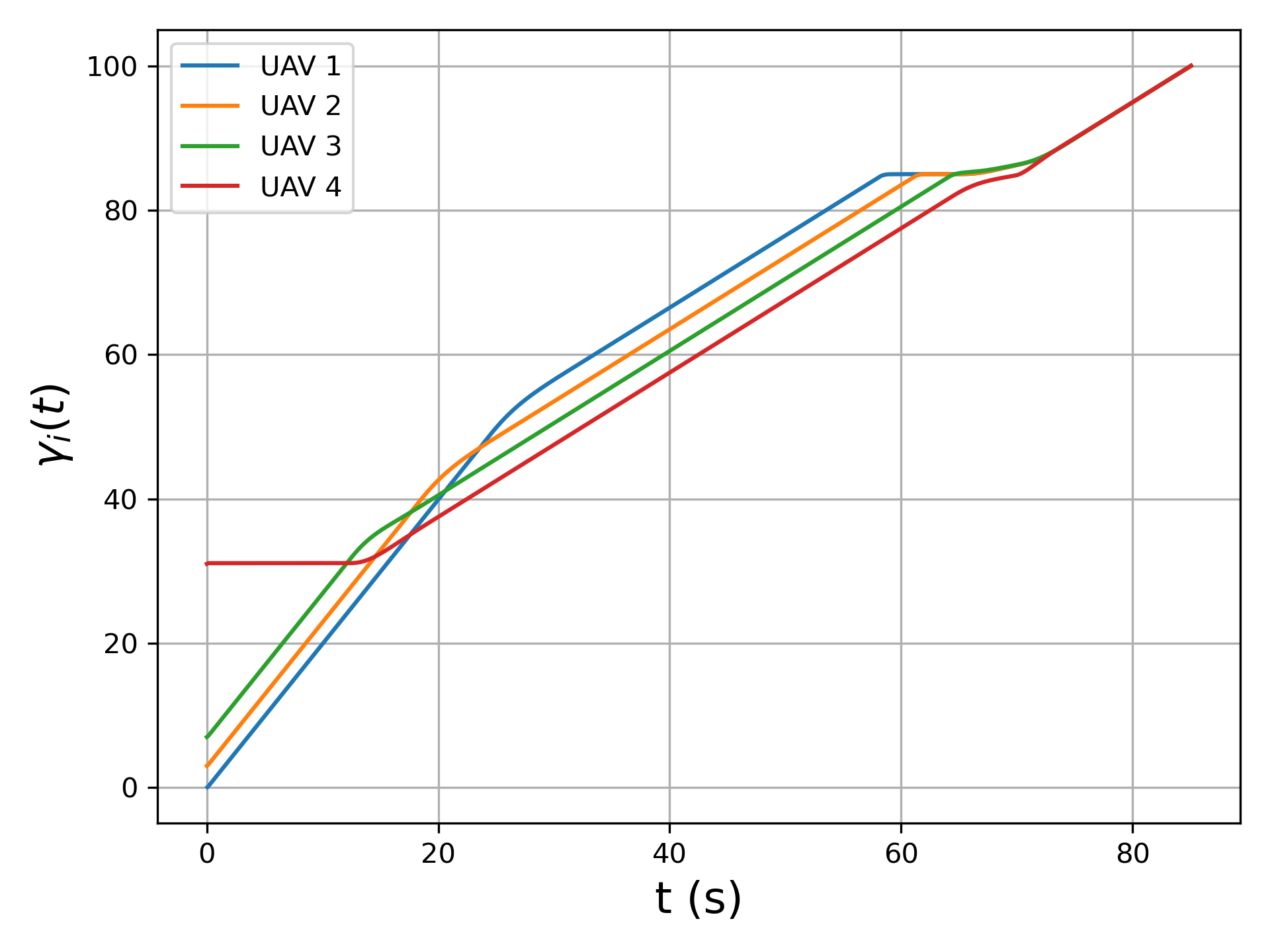} & \raisebox{1.5cm}{(a)} \\[4pt]
    \includegraphics[width=0.35\textwidth]{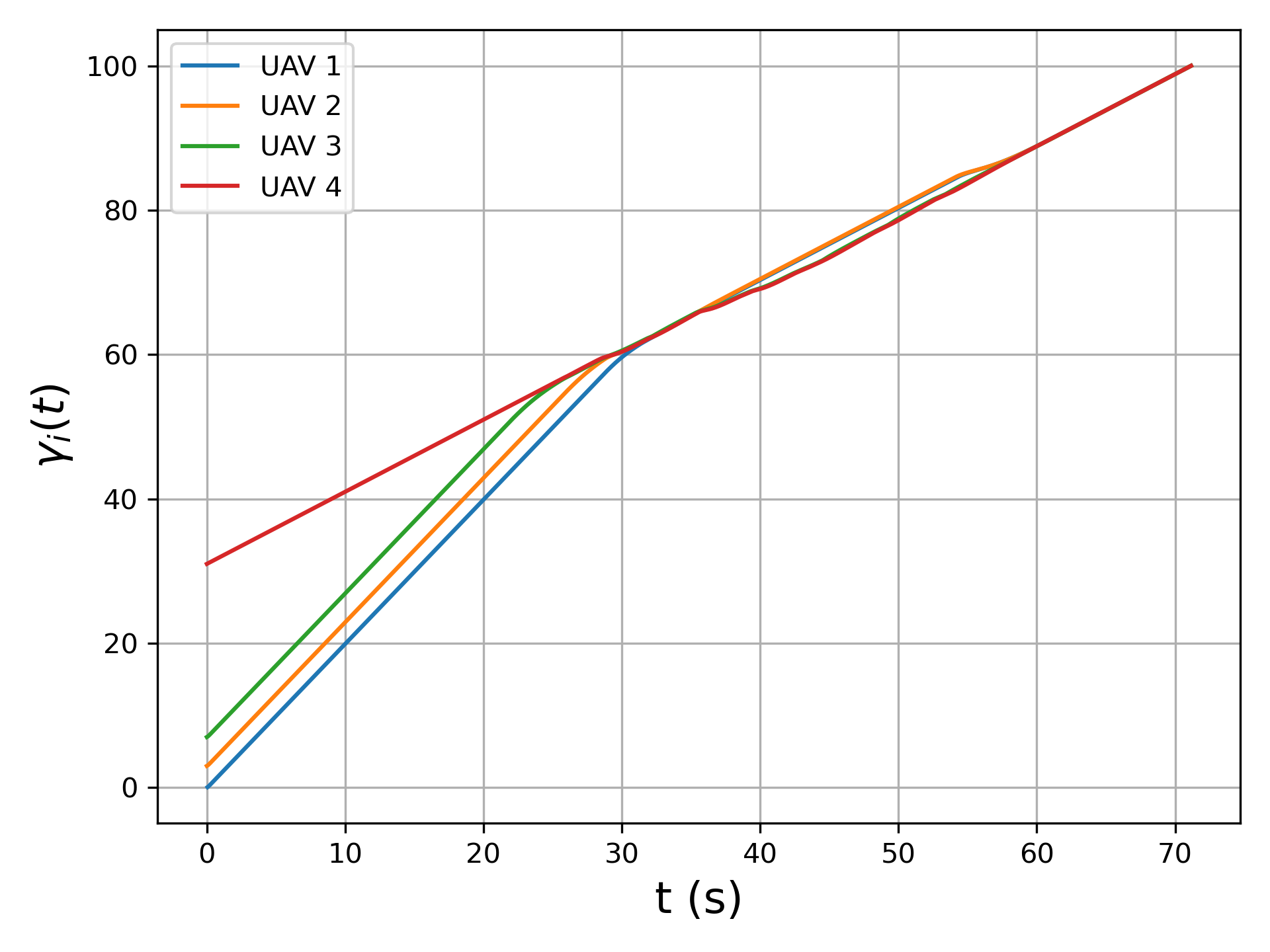}  & \raisebox{1.5cm}{(b)}
\end{tabular}
\caption{Consensus parameter $\gamma$ over time. (a) simulation under the cost function~\eqref{eq:Fi1}; 
    (b) simulation under the cost function~\eqref{eq:Fi2}.}
\label{fig:pair_one}
\end{figure}

The simulation is conducted with $\Delta = -3$, $\gamma_1^* = 0$, and $\gamma_2^* = 85$.

In this scenario, despite their initial spatial configuration, the UAVs are reordered due to disturbances as they approach the passage. In particular, the fourth UAV becomes closest to the bottleneck, while the first UAV is the furthest. As a result, the UAVs pass through the corridor according to the prescribed ordering determined by the index $i$ and the parameter $\Delta$.
As the $\gamma$ values suggest in Fig.~\ref{fig:pair_one}(a), the fourth UAV was 
successful in hovering and waiting for its scheduled turn to pass. However, UAV-1 failed to enter the corridor first; instead, the 3rd UAV entered 
first, followed by the 2nd. Consequently, reordering occurred already inside the corridor, resulting in three collisions between each pair of UAVs (see Fig.~\ref{fig:pair_two}(a)). The mission therefore fails, though the simulator allows the experiment to continue in the presence of collisions.

\subsubsection{Autonomous Ordering via Agent Competition/Game}
To address the limitations of the offline ordering approach, a game-theoretic formulation is introduced, in which the UAVs determine the passage order autonomously at the time of the mission rather than relying on a predefined ordering. 
The modified objective function again follows the same structure as in~\eqref{cost1}, 
with the exception that the second term is now defined as:
\vspace{-10pt}
\begin{equation}
\begin{split}
    F_i(\bm\delta^{k}_{i},\bm\delta^{k-1}_{-i}) = \sum_{\tau=1}^{K} \Biggl[
    (1-\mathcal{I}_{i\tau}^k)\sum_{j=1,\, j\neq i}^{N} 
    \left(\delta_{i\tau}^k - {\delta}_{j\tau}^{k-1}\right)^2 \\
    + \mathcal{I}_{i\tau}^k\sum_{j=1,\, j\neq i}^{N}  \Bigl[
    \max\!\left(0,\delta_{j\tau}^k - \delta_{i\tau}^{k-1}\right)^2 +\xi(\delta^k_{j\tau}-\delta^{k-1}_{i\tau})\psi_i
    \Bigr] \Biggr],
\end{split}
\label{eq:Fi2}
\end{equation}
where $\xi$ denotes the discrete Dirac delta (Kronecker delta), defined as
\[
    \xi(x) = \begin{cases} 1, & x = 0 \\ 0, & x \neq 0 \end{cases}
\]
and $\bm{\psi}$ represents a weighting vector that assigns relative importance to individual UAVs in the case where $\gamma_i = \gamma_j$ for all $i, j$. It serves as an artificial ordering mechanism to resolve ambiguity in equally favorable cases. It is required that all elements of $\bm{\psi}$ be mutually distinct, i.e. $\psi_i \neq \psi_j, \quad \forall\, i \neq j$. $\mathcal{I}_{i\tau}^k$ is as defined in~\eqref{eq:indicator}.

The coordination term is active outside the interval $[\gamma_1^*,\gamma_2^*]$, encouraging drones to 
coordinate. Inside the interval, the coordination term is replaced by a competitive term that penalizes drone $i$ for every neighbor that is ahead of it. This creates a racing incentive inside the corridor and ensuring successful execution of the mission.

The simulation parameters are set to $\gamma_1^* = 0$, $\gamma_2^* = 85$ as in the previous case, and 
$\bm{\psi} = [15, 25, 35, 45]$. As observed in Fig.~\ref{fig:pair_two}(b), the UAVs completed the mission without any collisions, successfully and autonomously determining the passage order, with the closest UAV to the bottleneck entering 
the corridor first, followed by the remaining ones accordingly. After exiting the 
tunnel, coordination is achieved (see Fig.~\ref{fig:pair_one}(b)), 
particularly once the competition term is deactivated (i.e., after $\gamma = 85$).

Thus, the offline ordering approach may fail depending on the disturbances, resulting in inter-agent collisions, whereas the game-theoretic approach successfully determines the passage 
order at the time of the mission regardless of the disturbances and ensures a safe mission execution. The game/competing cost function (e.g. see \eqref{eq:Fi2}) is therefore more reliable and better suited for real-world use.

\begin{figure}[ht]
\centering
\begin{tabular}{cl}
    \includegraphics[width=0.35\textwidth]{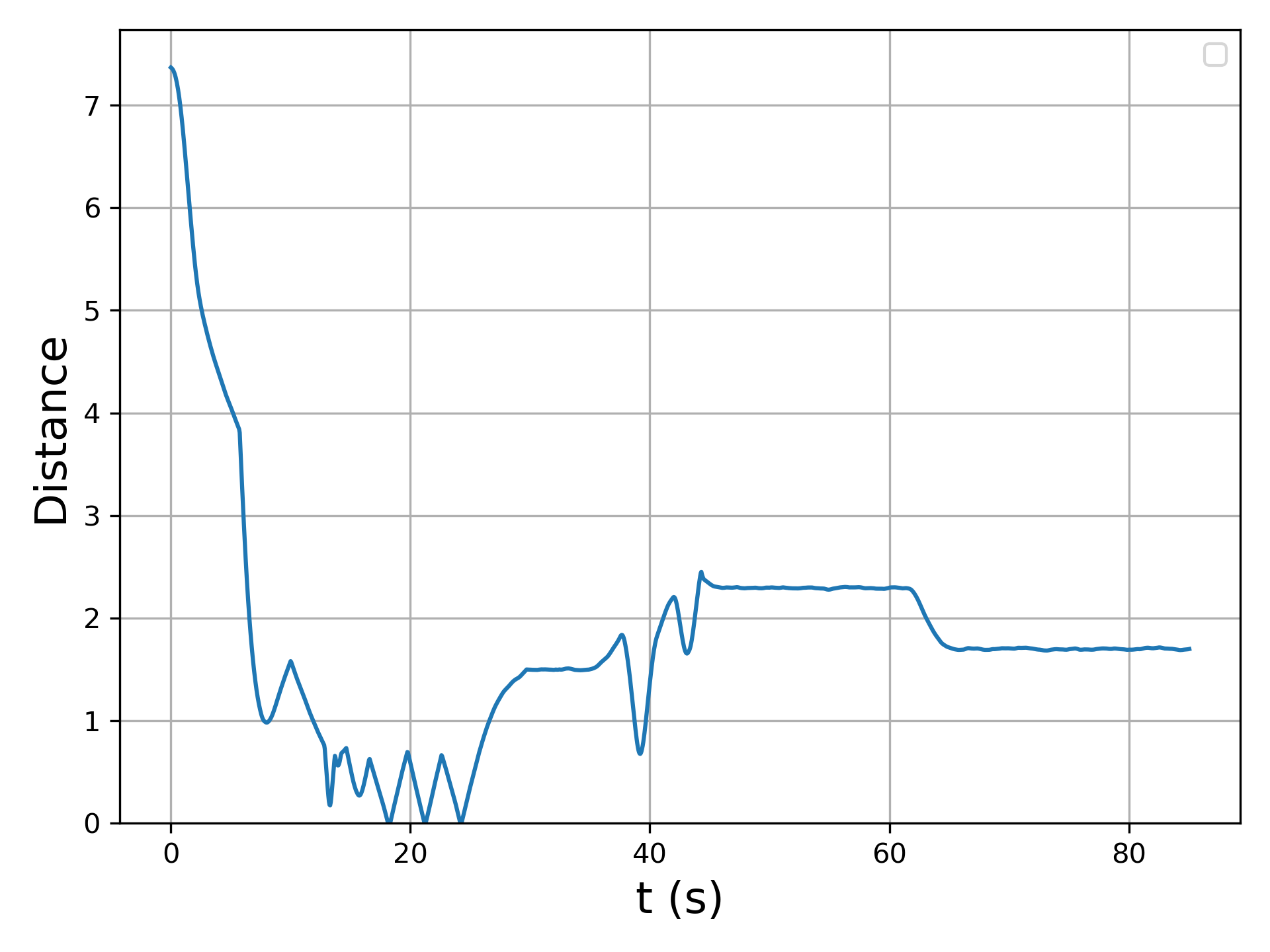} & \raisebox{1.5cm}{(a)} \\[4pt]
    \includegraphics[width=0.35\textwidth]{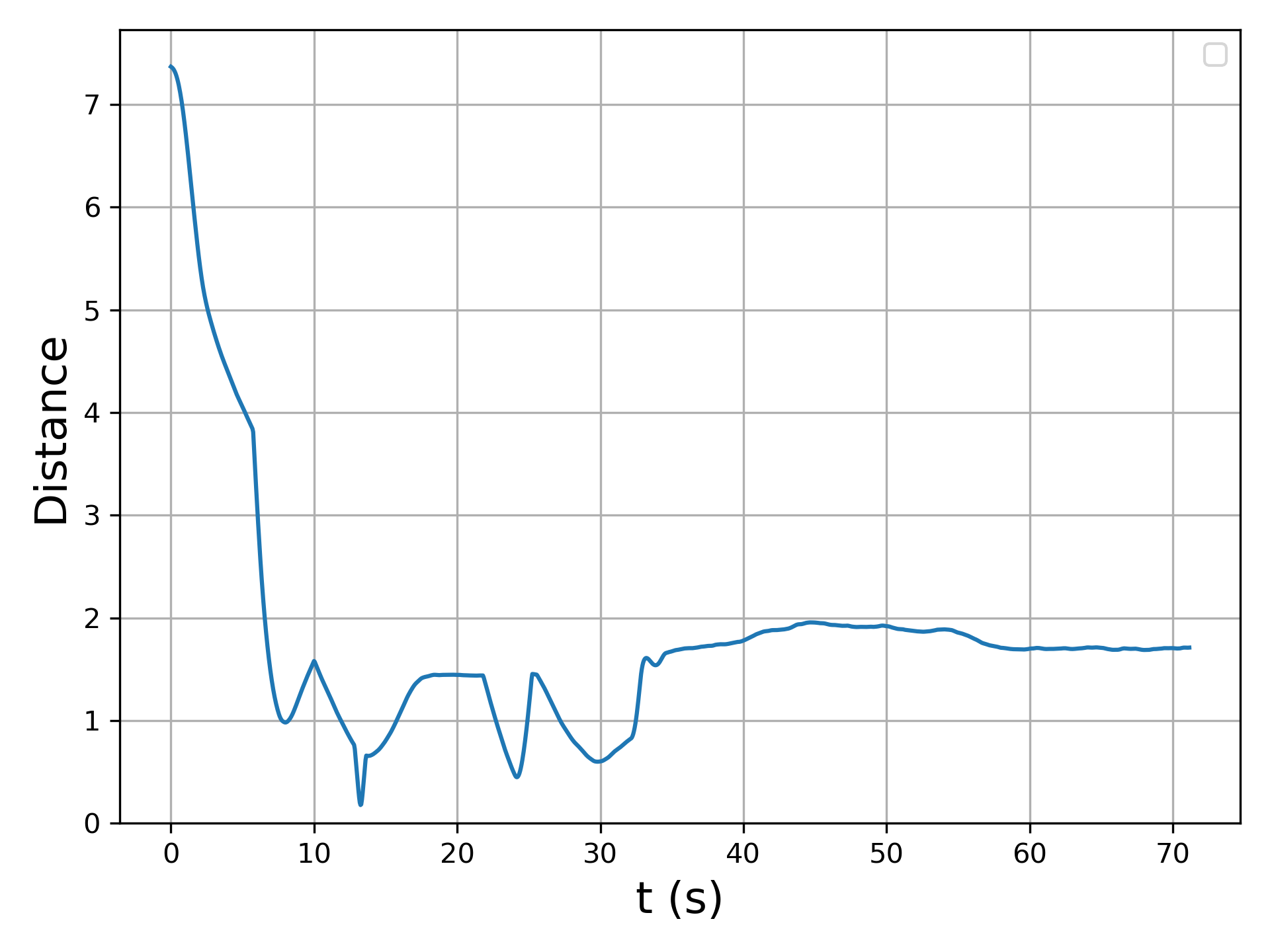}  & \raisebox{1.5cm}{(b)}
\end{tabular}
\caption{Minimum distance maintained between any two UAVs over time: (a) simulation under the cost function~\eqref{eq:Fi1}; 
    (b) simulation under the cost function~\eqref{eq:Fi2}.}
\label{fig:pair_two}
\end{figure}

The implementation code for all simulations presented in this work is publicly available 
at \url{https://github.com/amanucha/rotorpy_coordination}.

\section{Conclusion and Future Work}

In this paper, we analyzed a distributed model predictive control (DMPC) approach within the cooperative path-following framework.
We established exponential stability of the proposed discrete-time DMPC scheme under a fixed and connected communication network for the prediction horizon $K=1$. The analysis relies on an explicit characterization of the closed-loop dynamics and establishes contraction properties that ensure convergence while providing an explicit convergence rate.

Simulation results demonstrate the real-time applicability and agility of the proposed method. In particular, they highlight scalability with respect to the number of UAVs and illustrate the advantages of the approach in scenarios where preplanned coordination may fail. These results emphasize the role of the game/optimization-based structure in enabling adaptive and robust coordination, while also providing numerical validation of the theoretical findings.

Future work will focus on extending the theoretical analysis to time-varying communication networks and larger prediction horizons $K>1$, as well as to non-quadratic cost functions,  enlarging the range of mission settings. We also plan to pursue experimental validation on real UAV platforms.

\bibliographystyle{IEEEtran}
\bibliography{references}  

\end{document}